\documentclass{article}
\usepackage{amsmath,amsthm,amsfonts,amssymb}
\usepackage[colorlinks=true,citecolor=black,linkcolor=black,urlcolor=blue]{hyperref}
\usepackage[numbers,sort&compress]{natbib}
\allowdisplaybreaks
\theoremstyle{plain}
\newtheorem{thm}{Theorem}
\newtheorem{lem}[thm]{Lemma}
\newcommand{\ceil}[1]{\lceil{#1}\rceil}
\newcommand{\arXiv}[1]{arXiv:\,\href{http://arxiv.org/abs/#1}{#1}}
\newcommand{\msn}[1]{MR:\,\href{http://www.ams.org/mathscinet-getitem?mr=MR#1}{#1}}
\newcommand{\doi}[1]{doi:\,\href{http://dx.doi.org/#1}{#1}}

\begin{document}
\renewcommand{\thefootnote}{\fnsymbol{footnote}}	
\title{\bf Hypergraph Colouring and Degeneracy}
\author{David~R.~Wood\,\footnotemark[1]}
\date{10 October 2013; revised \today}
\maketitle

\begin{abstract}
A hypergraph is \emph{$d$-degenerate} if every subhypergraph has a vertex of degree at most $d$. A greedy algorithm colours every such hypergraph with at most $d+1$ colours. We show that this bound is tight, by constructing an $r$-uniform $d$-degenerate hypergraph with chromatic number $d+1$ for all $r\geq2$ and $d\geq1$. Moreover, the hypergraph is triangle-free, where a \emph{triangle} in an $r$-uniform hypergraph consists of three edges whose union is a set of $r+1$ vertices. 
\end{abstract}

\footnotetext[1]{School of Mathematical Sciences, Monash University, Melbourne, Australia
  (\texttt{david.wood@monash.edu}). Research supported by  the Australian Research Council.}

\renewcommand{\thefootnote}{\arabic{footnote}}

\section{Introduction}

\citet{EL75} proved the following fundamental result about colouring hypergraphs\footnote{A  \emph{hypergraph} $G$ consists of a set $V(G)$ of \emph{vertices} and a set $E(G)$ of subsets of $V(G)$ called \emph{edges}. A hypergraph is \emph{$r$-uniform} if every edge has size $r$. A \emph{graph} is a $2$-uniform hypergraph. A hypergraph $H$ is a \emph{subhypergraph} of a hypergraph $G$ if $V(H)\subseteq V(G)$ and $E(H)\subseteq E(G)$. A \emph{colouring} of a hypergraph $G$ assigns one colour to each vertex in $V(G)$ such that no edge in $E(G)$ is monochromatic. The \emph{chromatic number} of $G$, denoted by $\chi(G)$, is the minimum number of colours in a colouring of $G$. A colouring of $G$ can be thought of as a partition of $V(G)$ into  \emph{independent sets}, each containing no edge. The \emph{degree} of a vertex $v$ is the number of edges that contain $v$. See the textbook of \citet{BergeBook} for other notions of degree in a hypergraph.}

\begin{thm}[\citep{EL75}]
\label{EL75}
For fixed $r$, every $r$-uniform hypergraph with maximum degree $\Delta$ has chromatic number at most $O(\Delta^{1/(r-1)})$.
\end{thm}

Theorem~\ref{EL75} implies that every $r$-uniform hypergraph with maximum degree $\Delta$ has an independent set of size at least  $\Omega(n/\Delta^{1/(r-1)})$. \citet{Spencer} proved the following stronger bound.

\begin{thm}[\citep{Spencer}]
\label{Spencer}
For fixed $r$, every $r$-uniform hypergraph with $n$ vertices and average degree $d$ has an independent set of size at least $\Omega(n/d^{1/(r-1)})$.
\end{thm}

A hypergraph is \emph{$d$-degenerate} if every subhypergraph has a vertex of degree at most $d$. A minimum-degree-greedy algorithm colours every $d$-degenerate hypergraph with at most $d+1$ colours. This bound is tight for graphs ($r=2$) since the complete graph on $d+1$ vertices is $d$-degenerate, and of course, has chromatic number $d+1$. However, this observation does not generalise for $r\geq3$. In particular, for the complete $r$-uniform hypergraph on $n$ vertices, every vertex has degree $\binom{n-1}{r-1}$, yet the chromatic number is $\ceil{\frac{n}{r-1}}$. Thus for $r\geq 3$, the degeneracy is much greater than the chromatic number.

Given Theorems~\ref{EL75} and \ref{Spencer}, it seems plausible that for $r\geq 3$, every $r$-uniform $d$-degenerate hypergraph is $o(d)$-colourable. It even seems possible that  every $r$-uniform $d$-degenerate hypergraph is $O(d^{1/(r-1)})$-colourable. This  natural strengthening of Theorems~\ref{EL75} and \ref{Spencer} would (roughly) say that $G$ can be partitioned into independent sets, whose average size is that guaranteed by Theorem~\ref{Spencer}. 

This note rules out these possibilities, by showing that the naive upper bound  $\chi\leq d+1$ is tight for all $r$. This is the main conclusion of this paper. Moreover, we prove it for triangle-free hypergraphs, where a \emph{triangle} in an $r$-uniform hypergraph consists of three edges whose union is a set of $r+1$ vertices. Observe that this definition with $r=2$ is equivalent to the standard notion of a triangle in a graph (although there are other notions of a triangle in a hypergraph \citep{CM13}). 

\begin{thm}
\label{Main}
For all $r\geq2$ and $d\geq 1$ there is a triangle-free $d$-degenerate $r$-uniform hypergraph with chromatic number $d+1$.
\end{thm}

Theorem~\ref{Main} and its proof is a generalisation of a result of \citet{AKS99} who proved it for graphs ($r=2$). Of course, the complete graph $K_{d+1}$ is $d$-degenerate with chromatic number $d+1$. The triangle-free property was the main conclusion of their result. See \citep{KR-RSA10,Alon85} for other related results.

\section{Proof}

Theorem~\ref{Main} is a corollary of the following:

\begin{lem}
\label{MainMain}
Fix $r\geq 2$. For all $d\geq 1$ there is a triangle-free $d$-degenerate $r$-uniform hypergraph $G_d$ with chromatic number $d+1$, such that in every $(d+1)$-colouring of $G_d$ each colour is assigned to at least $r-1$ vertices.
\end{lem}

\begin{proof}
We proceed by induction on $d$. First consider the base case $d=1$. Let $n:=r(r-1)$. Let $V(G_1):=\{v_1,\dots,v_n\}$ and $E(G_1):=\{e_i:1\leq i\leq n-r+1\}$, where $e_i:=\{v_i,v_{i+1},\dots,v_{i+r-1}\}$. If $S\subseteq V(G_1)$ and $i$ is minimum such that $v_i\in S$, then $v_i$ has degree at most 1 in the subhypergraph  induced by $S$. Thus $G_1$ is $1$-degenerate. If $e_i,e_j,e_k$ are three edges in $G_1$ with $i<j<k$, then $e_i\cup e_j\cup e_k$ includes the $r+2$ distinct vertices $v_i,v_{i+1},\dots,v_{i+r-1},v_{j+r-1},v_{k+r-1}.$ Hence $G_1$ is triangle-free. Consider a 2-colouring of $G_1$. Clearly, $G_1$ contains $r-1$ pairwise disjoint edges, each of which contains vertices of both colours. Hence each colour is assigned to at least $r-1$ vertices. This completes the base case. 

Now assume that $G_{d-1}$ is a triangle-free $(d-1)$-degenerate $r$-uniform hypergraph with chromatic number $d$, such that in every $d$-colouring of $G_{d-1}$ each colour is assigned to at least $r-1$ vertices.

Initialise $G_d$ to consist of $d+r-2$ disjoint copies $H_1,\dots,H_{d+r-2}$ of $G_{d-1}$. Let $S$ be a set of $(r-1)d$ vertices in $H_1\cup\dots\cup H_{d+r-2}$ such that $|S\cap V(H_i)|\in\{0,r-1\}$ for $1\leq i\leq d+r-2$. That is, $S$ contains exactly  $r-1$ vertices from exactly $d$ of the $H_i$, and contains no vertices from the other  $r-2$. Now, for each such set $S$, add $r-1$ \emph{new} vertices $v_1,\dots,v_{r-1}$ to $G_d$ and add the \emph{new} edge $(S\cap V(H_i))\cup\{v_j\}$ to $G_d$ whenever $|S\cap V(H_i)|=r-1$. Thus each new vertex has degree $d$. Since $H_1\cup \dots\cup H_{d+r-2}$ is $d$-degenerate, $G_d$ is also $d$-degenerate.

Suppose on the contrary that $G_d$ contains a triangle $T$. Since $G_{d-1}$ is triangle-free, at least one edge in $T$ is a new edge, which is contained in $V(H_i)\cup\{v\}$ for some $i\in[1,d+r-2]$ and some new vertex $v$. Each vertex in a triangle is in at least two of the edges of the triangle. However, by construction,  $v$ is contained in only one edge contained in $V(H_i)\cup\{v\}$. Thus $G_d$ is triangle-free. 

Since  $H_1\cup\dots\cup H_{d+r-2}$ is $d$-colourable, and no edge contains only new vertices, assigning all the new vertices a $(d+1)$-th colour produces a $(d+1)$-colouring of $G_d$. Thus $\chi(G_d)\leq d+1$. 

Suppose on the contrary that $G_d$ has a $(d+1)$-colouring with at most $r-2$ vertices of some colour, say `blue'. Say the other colours are $1,\dots,d$. At most $r-2$ copies of the $H_i$ contain blue vertices. Hence, without loss of generality, $H_1,\dots,H_d$ contain no blue vertices. That is, $H_1,\dots,H_d$ are $d$-coloured with colours $1,\dots,d$. By induction,  $H_i$ contains a set $S_i$ of $r-1$ vertices coloured $i$ for $1\leq i\leq d$. By construction, there are $r-1$ vertices $v_1,\dots,v_{r-1}$ in $G_d$, such that $S_i\cup\{v_j\}$ is an edge of $G_d$ for $1\leq i\leq d$ and $1\leq j\leq r-1$. Since each such edge is not monochromatic, each vertex $v_j$ is coloured blue. In particular, there are at least $r-1$ blue vertices, which is a contradiction. Therefore, in every  $(d+1)$-colouring of $G_d$, each  colour class has  at least $r-1$ vertices, as claimed. (In particular, $G_d$ has no $d$-colouring.) 
\end{proof}

\section{An Open Problem}

We conclude with an open problem. The \emph{girth} of a graph (that contains some cycle) is the length of its shortest cycle.  \citet{Erdos59} proved that there exists a graph with chromatic number at least $k$ and girth at least $g$, for all $k\geq 3$ and $g\geq 4$. (\citet{EH66} proved an analogous result for hypergraphs). Theorem~\ref{Main} strengthens this result for triangle-free graphs (that is, with girth $g=4$). This leads to the following question: Does there exist a  $d$-degenerate graph with chromatic number $d+1$ and girth $g$, for all $d\geq 2$ and $g\geq4$? Odd cycles prove the $d=2$ case. An affirmative answer would strengthen the above result of \citet{Erdos59}. A negative answer would also be interesting---this would provide a non-trivial upper bound on the chromatic number of $d$-degenerate graphs with girth $g$. 

\subsection*{Note} After this paper was written the author discovered the beautiful paper by \citet{KN99} which proves a strengthening of Theorem~\ref{Main} and includes the positive solution of the above open problem.

\subsection*{Acknowledgement} Thanks to an anonymous referee for pointing out an error in an earlier version of this paper. 


\def\soft#1{\leavevmode\setbox0=\hbox{h}\dimen7=\ht0\advance \dimen7
  by-1ex\relax\if t#1\relax\rlap{\raise.6\dimen7
  \hbox{\kern.3ex\char'47}}#1\relax\else\if T#1\relax
  \rlap{\raise.5\dimen7\hbox{\kern1.3ex\char'47}}#1\relax \else\if
  d#1\relax\rlap{\raise.5\dimen7\hbox{\kern.9ex \char'47}}#1\relax\else\if
  D#1\relax\rlap{\raise.5\dimen7 \hbox{\kern1.4ex\char'47}}#1\relax\else\if
  l#1\relax \rlap{\raise.5\dimen7\hbox{\kern.4ex\char'47}}#1\relax \else\if
  L#1\relax\rlap{\raise.5\dimen7\hbox{\kern.7ex
  \char'47}}#1\relax\else\message{accent \string\soft \space #1 not
  defined!}#1\relax\fi\fi\fi\fi\fi\fi}

\end{document}